\documentclass[12pt]{amsart}
\usepackage{amsfonts,amssymb,amscd,amsmath,enumerate,verbatim}
\usepackage[latin1]{inputenc}
\usepackage{amscd}
\usepackage{latexsym}

\usepackage[dvipdfmx]{graphicx}
\usepackage{mathptmx}

\def\NZQ{\mathbb}               

\def\ZZ{{\NZQ Z}}
\def\RR{{\NZQ R}}

\def\frk{\mathfrak}               

\def\Phi{{\frk N}}
\def\ab{{\mathbf a}}

\def\wb{{\mathbf w}}
\def\xb{{\mathbf x}}

\def\Ib{{\mathbf I}}

\def\opn#1#2{\def#1{\operatorname{#2}}} 
\opn\ini{in} \opn\sgn{sgn} \opn\Gr{Gr} \opn\Im{Im}
\opn\gr{gr}

\def\Ac{{\mathcal A}}

\def\Mc{{\mathcal M}}

\def\Gc{{\mathcal G}}

\newtheorem{Theorem}{Theorem}[section]
\newtheorem{Lemma}[Theorem]{Lemma}
\newtheorem{Corollary}[Theorem]{Corollary}
\newtheorem{Proposition}[Theorem]{Proposition}

\theoremstyle{definition}
\newtheorem{Remark}[Theorem]{Remark}

\newtheorem{Example}[Theorem]{Example}

\newtheorem{Definition}[Theorem]{Definition}

\newtheorem{Question}[Theorem]{Question}

\let\epsilon\varepsilon
\let\phi=\varphi
\let\kappa=\varkappa
\opn\dis{dis}
\opn\height{height}
\opn\dist{dist}
\def\pnt{{\raise0.5mm\hbox{\large\bf.}}}

\opn\Lex{Lex}
\opn\conv{conv}

\begin{document}

\title{Quadratic Gr\"obner bases of block diagonal matching field ideals and toric degenerations of Grassmannians}

\author{Akihiro Higashitani and Hidefumi Ohsugi}

\address{Akihiro Higashitani,
Department of Pure and Applied Mathematics,
Graduate School of Information Science and Technology,
Osaka University,
Suita, Osaka 565-0871, Japan}
\email{higashitani@ist.osaka-u.ac.jp}

\address{Hidefumi Ohsugi,
	Department of Mathematical Sciences,
	School of Science, 
	Kwansei Gakuin University,
	Sanda, Hyogo 669-1337, Japan} 
\email{ohsugi@kwansei.ac.jp}


\begin{abstract}
In the present paper, we prove that the toric ideals of certain $s$-block diagonal matching fields 
have quadratic Gr\"obner bases. Thus, in particular, those are quadratically generated. 
By using this result, we provide a new family of toric degenerations of Grassmannians. 
\end{abstract}

\maketitle

\section{Introduction}

A toric degeneration of a given projective variety $X$ is a flat family of varieties whose central fiber is a toric variety $X_0$ 
and all of whose general fibers are isomorphic to $X$. The resulting toric variety $X_0$ has a rich information on the original variety $X$. 
Hence, providing a toric degeneration is a useful tool to analyze algebraic varieties by using toric geometry. 
In the present paper, we study a new family of toric degeneration of Grassmannians. 

Let $K$ be a field. We use the notation $\Gr(r,n)$ for the Grassmannian, 
which is the space of $r$-dimensional subspaces of the $n$-dimensional vector space $K^n$. 
The studies of toric degenerations of Grassmannians, flag varieties, Schubert varieties and Richardson varieties 
are an active research area in various branches of mathematics, 
such as algebraic geometry, algebraic topology, representation theory, commutative algebra, combinatorics, and so on. 
For the study of toric degenerations, see, e.g. \cite{BFFHL, BLMM, BMN, CHM, CM, MS, RW} and so on. 
There are many ways to construct toric degenerations for them. 
The main example of toric degenerations of $\Gr(r,n)$ is, so-called, the {\em Gelfand--Tsetlin degeneration} (see \cite[Section 14]{MillerSturmfels}). 
Moreover, the theory of Newton--Okounkov bodies can be applied to provide toric degenerations of varieties (see, e.g., \cite{HaradaKaveh} and the references therein). 
We can construct Newton--Okounkov bodies of Grassmannians from plabic graphs, which are certain bipartite graphs drawn in the disc, by a combinatorial manner (\cite{RW}). 
Furthermore, tropical geometry is also used for the study of toric degenerations of defining ideals in general (\cite{KavehManon}). 
Namely, the top-dimensional cones of tropicalizations of varieties are good candidates to give toric degenerations, see \cite[Lemma 1]{BMN} and also \cite{KavehManon}.
It was also prove in \cite{FR} that the tropical cones arising from matching fields are a subfamily of Stiefel tropical linear spaces. 

In addition, the theory of SAGBI bases can be also applied to provide toric degenerations. This is the tool we will use in the present paper. 
We refer the reader to \cite[Section 11]{Stu} for the introduction to the theory of SAGBI bases. 
For specifying the monomial order for SAGBI bases, we use {\em matching fields} in the sense of \cite{SZ}, 
which we will explain in Section~\ref{sec:prepare}. 
Note that the same approach of using matching fields to obtain toric degenerations of Schubert and Richardson varieties 
has been studied in \cite{CM_Sch} and \cite{BCM}, respectively. 

Let us review the previous known results from \cite{CM, MS} which are strongly related to our main theorems. 
In \cite{MS}, Mohammadi and Shaw discuss a necessary condition for matching fields to provide a toric degeneration of $\Gr(r,n)$ as follows: 
\begin{Theorem}[{See \cite[Theorems 1.2 and 1.3]{MS}}]
If the matching field $\Lambda$ of $\Gr(r,n)$ produces its toric degeneration, then $\Lambda$ is non-hexagonal. 
Moreover, the converse is also true for $r=3$ under the assumption that the matching field ideal is quadratically generated. 
\end{Theorem}
\noindent
(The terminologies will be defined in Section~\ref{sec:prepare}.) 
Hence, for the discussion of the existence of toric degenerations of $\Gr(3,n)$ from matching fields, 
the quadratic generation of the matching field ideals is quite important. 
It is mentioned in \cite[Example 3.12 and Remark 3.13]{MS} that the matching field ideals are not necessarily quadratically generated. 
Thus, identifying families of matching fields with quadratic ideals is a natural problem. 
As a nice class of matching fields, $s$-block diagonal matching fields are introduced to be expected that 
their corresponding ideals are quadratically generated (\cite[Definition 4.1]{MS}), and the following is proved: 
\begin{Theorem}[{\cite[Theorem 1.4 and Corollary 1.5]{MS}}]\label{thm:MS}
The ideals of $2$-block diagonal matching fields for $\Gr(3,n)$ are quadratically generated. 
Moreover, every $2$-block diagonal matching field gives rise to a toric degeneration of $\Gr(3,n)$. 
\end{Theorem}
Furthermore, in \cite{CM}, Theorem~\ref{thm:MS} is generalized for any $r$. Namely, 
\begin{Theorem}[{\cite[Theorems 4.1 and 4.3]{CM}}]\label{thm:CM}
The ideals of $2$-block diagonal matching fields for $\Gr(r,n)$ are quadratically generated. 
Moreover, every $2$-block diagonal matching field gives rise to a toric degeneration of $\Gr(r,n)$. 
\end{Theorem}

Taking those theorems into account, we prove the following main results of the present paper: 
\begin{Theorem}[{See Theorems~\ref{thm:quad} and \ref{thm:toric_deg}}]\label{thm:main}
Given $\ab=(a_1,\ldots,a_s) \in \ZZ_{>0}^s$ with $\sum_{i=1}^s a_i=n$ and $a_i \in \{1,2\}$ for $i \not\in\{1,s\}$, 
consider the $s$-block diagonal matching field $\Lambda_\ab$. 
Then the matching field ideal $J_{\Lambda_\ab}$ has a quadratic Gr\"obner basis. 
In particular, $J_{\Lambda_\ab}$ is quadratically generated. 

Moreover, the generating set $\{\det(\xb_I) : I \in \Ib_{r,n}\}$ of the Pl\"ucker algebra $\Ac_{r,n}$ forms a SAGBI basis for $\Ac_{r,n}$ 
with respect to the weight matrix $M_\ab$ associated with $\Lambda_\ab$. 
Consequently, every $\Lambda_\ab$ gives rise to a toric degeneration of $\Gr(r,n)$. 
\end{Theorem}
We notice that Theorem~\ref{thm:main} directly implies Theorem~\ref{thm:CM}. 
Moreover, as we can see in Sections~\ref{sec:qGB} and \ref{sec:toric_deg}, the proofs become a little simpler than those of Theorem~\ref{thm:CM}.

The paper is organized as follows: 
In Section~\ref{sec:prepare}, we prepare the necessary terminologies and the notation we will use. 
In Section~\ref{sec:graph}, for the proof of our theorem, we recall the theory of toric ideals of the edge rings and their Gr\"obner bases. 
In Section~\ref{sec:qGB}, we will give a proof of the first half part of Theorem~\ref{thm:main} (Theorem~\ref{thm:quad}). 
In Section~\ref{sec:toric_deg}, we will give a proof of the second half part of Theorem~\ref{thm:main} (Theorem~\ref{thm:toric_deg}). 

\medskip

\section*{Acknowledgements}
The authors are partially supported by JSPS KAKENHI $\sharp$18H01134 and $\sharp$20K03513. 

\medskip


\section{Pl\"ucker algebra and $s$-Block diagonal matching fields}\label{sec:prepare}

Given integers $r$ and $n$ with $1 < r < n$, let $\Ib_{r,n}$ be the set of all $r$-subsets of $[n]=\{1,2,\ldots,n\}$. 
Let $S = K[P_I : I \in  \Ib_{r,n}]$ be the polynomial ring with $\binom{n}{r}$ variables. 
Let $\xb=(x_{ij})_{1 \leq i \leq r, 1 \leq j \leq n}$ be the $r \times n$ matrix of variables 
and let $R=K[\xb]$ be the polynomial ring with $rn$ variables. 
The {\em Pl\"ucker ideal} $I_{r,n}$ is defined by the kernel of the ring homomorphism 
$$\psi:S \rightarrow R, \;\; P_I \mapsto \det(\xb_I),$$
where $\xb_I$ denotes the $r \times r$ submatrix of $\xb$ whose columns are indexed by $I$. 
The {\em Pl\"ucker algebra} $\Ac_{r,n}$ is the image $\Im(\psi)$ of this map, which is isomorphic to $S/I_{r,n}$. 
The Pl\"ucker algebra $\Ac_{r,n}$ is well-known to be a homogeneous coordinate ring of the Pl\"ucker embedding of the Grassmannians $\Gr(r,n)$, 
called the {\em Pl\"ucker embedding}.

Let ${\mathfrak S}_r$ denote the symmetric group on $[r]$. 
An $r \times n$ {\em matching field} is a map $\Lambda : \Ib_{r,n} \rightarrow {\mathfrak S}_r$. 
For $I=\{i_1,\ldots,i_r\} \in \Ib_{r,n}$ with $1 \leq i_1 < \cdots < i_r \leq n$ and a matching field $\Lambda$, we associate the monomial of $R$ 
$$\xb_{\Lambda(I)}:=x_{\sigma(1) i_1} \cdots x_{\sigma(r) i_r},$$
where $\sigma=\Lambda(I) \in {\mathfrak S}_r$. 
We define a ring homomorphism $$\psi_\Lambda : S \rightarrow R, \;\;\psi_\Lambda(P_I) = \sgn(\Lambda(I))\xb_{\Lambda(I)},$$ 
where $\sgn(\sigma)$ denotes the signature of $\sigma \in {\mathfrak S}_r$. 
Then the {\em matching field ideal} $J_\Lambda$ of $\Lambda$ is the kernel of $\psi_\Lambda$.

\begin{Definition}[{Coherent matching fields (\cite[Section 1]{SZ}, \cite[Definition 2.8]{MS})}]
A matching field $\Lambda$ is said to be {\em coherent} if there exists an $r \times n$ matrix $M \in \RR^{r \times n}$ with its entries in $\RR$ 
such that for every $I \in \Ib_{r,n}$ the initial form $\ini_M( \det(\xb_I) )$ of $\det(\xb_I)$ with respect to $M$, 
which is the sum of all terms of $\det(\xb_I)$ having the lowest weights, is equal to $\psi_\Lambda(P_I)$. 
In this case, we call $\Lambda$ a coherent matching field {\em induced by $M$}. 
For the matching field ideals, we use the notation $J_M$ instead of $J_\Lambda$ if $\Lambda$ is a coherent matching field induced by $M$. 
\end{Definition}
In the original definition \cite[Section 1]{SZ} of coherent matching fields, 
the initial form is set to be the sum of all terms having the {\em highest} weights, 
but we usually employ the definition with the sum of all terms having the {\em lowest} weights 
when it is related to the context of tropical geometry, following the convention of our main reference \cite{MS}. 

The main object of the present paper is the following matching fields: 
\begin{Definition}[{\cite[Definition 4.1]{MS}, $s$-block diagonal matching fields}]\label{def:block}
Let $\ab = (a_1,\dots, a_s) \in \ZZ_{>0}^s$ such that $\sum_{i=1}^s a_i = n$. For $k=1,2,\dots , s$, let 
$$I_k=\{\alpha_{k-1}+1 ,\alpha_{k-1}+ 2,  \ldots, \alpha_k \}=[\alpha_k] \setminus [\alpha_{k-1}],$$
where $\alpha_0=0$ and $\alpha_k= \sum_{i=1}^k a_i$.
Note that $\alpha_s =n$.
Then the $s$-{\em block diagonal matching field} $\Lambda_\ab$ associated with $\ab$ is defined by 
\[\Lambda_\ab (I) = \left\{
\begin{array}{cc}
(1 \ 2) & \mbox{if } |I \cap I_q| = 1 \mbox{ where } q = \min \{ t : I_t \cap I \neq \emptyset\}, \\
 {\rm id} & \mbox{otherwise.}
\end{array}
\right.\]
\end{Definition}
Any $s$-block diagonal matching field $\Lambda_\ab$ is coherent.
In fact, $\Lambda_\ab$ is induced by the following weight matrix $M_\ab$:
\[
\left(
\begin{array}{cccc|ccc|c|cccc}
0 & 0& \cdots & 0 &0 & \cdots& 0&\cdots&0&\cdots&0 & 0\\
\alpha_1 & \alpha_1 -1 & \cdots & 1 & 
\alpha_2 &  \cdots & \alpha_1 +1 & 
\cdots &
\alpha_s &  \cdots & \alpha_{s-1} +2& \alpha_{s-1} +1\\
n \beta  & (n-1) \beta &\cdots &&&&&& &\cdots& 2\beta& \beta\\
 \vdots &\vdots &   &   & &  &    & & & & \vdots & \vdots \\
n \beta^{r-2}  & (n-1) \beta^{r-2} &\cdots &&&&&& &\cdots& 2\beta^{r-2}& \beta^{r-2}\\
\end{array} \right),
\]
where $\beta \gg 0$.

\begin{Example}\label{ex:block}
If $r=3$, $n=9$ and $\ab=(2,2,2,3)$, then 
\[
M_\ab=
\begin{pmatrix}
0 & 0 & 0 & 0 & 0 & 0 & 0 & 0 & 0\\
2 & 1 & 4 & 3 & 6 & 5 & 9 & 8 & 7\\
900 & 800 & 700 & 600 & 500 & 400 & 300 & 200 & 100
\end{pmatrix}.
\]
On the other hand, if $r=4$, $n=9$ and $\ab=(5,4)$, then 
\[
M_\ab =
\begin{pmatrix}
0 & 0 & 0 & 0 & 0 & 0 & 0 & 0 & 0\\
5 & 4 & 3 & 2 & 1 & 9 & 8 & 7 & 6\\
900 & 800 & 700 & 600 & 500 & 400 & 300 & 200 & 100 \\
90000 &80000 &70000 &60000 &50000 &40000 &30000 &20000 &10000
\end{pmatrix}.
\]
\end{Example}

\begin{Remark}
(a) When $\ab=(n)$, i.e., $s=1$, 
the corresponding block diagonal matching field is so-called the {\em diagonal matching field} (see \cite[Example 1.3]{SZ}). 
This actually gives rise to the Gelfand-Tsetlin degeneration (see \cite[Section 14]{MillerSturmfels}).

(b) In \cite{CM}, the terminology ``block diagonal'' is used for ``$2$-block diagonal'' in the sense of Definition~\ref{def:block}. 
Namely, the first example of Example~\ref{ex:block} is not block diagonal in the sense of \cite{CM}. 
\end{Remark}

Let $\wb_\ab = (w_1,\ldots, w_n)$ be the second row of $M_\ab$.
Note that, for $I=\{i_1,\ldots,i_r\}$ with $1 \leq i_1 < \cdots < i_r \leq n$, we have
\[
\Lambda_\ab (I) = (1 \ 2) \  \Longleftrightarrow \ w_{i_1} < w_{i_2}.
\]
Thus $\wb_\ab$ is useful to study $\Lambda_\ab$. Moreover $\wb_\ab$ satisfies
\begin{equation}
\label{b_d_condition}
i < j \mbox{ and } w_i > w_j \Longrightarrow w_i  >  w_{i+1} > \cdots > w_{j-1} > w_j.
\end{equation}

In the present paper, we mainly consider $s$-block diagonal matching fields $\Lambda_\ab$ such that $a_i \in \{1,2\}$ for all $i \notin \{1,s\}$.
This is equivalent to the condition 
\begin{equation}\label{b_d_condition2}
i < j \mbox{ and  } w_i < w_j > w_{j+1} > w_{j+2}
\Longrightarrow w_j> w_{j+1} > \cdots > w_n.
\end{equation}
In particular, if $s=2$, then the condition \eqref{b_d_condition2} is always satisfied.

\medskip


\section{Edge rings of bipartite graphs}\label{sec:graph}

In this section, we recall the notion of edge rings of bipartite graphs. 
We will see that the matching field ideal of a $2 \times n$ matching field $\Lambda$ 
is the toric ideal of a certain bipartite graph associated with $\Lambda$. 
We will also discuss Gr\"obner bases of matching field ideals. 
Consult \cite[Chapter~5]{binomialideals} and \cite{Vbook} for the introduction to the toric ideals of graphs and 
\cite{binomialideals, Stu} for the introduction to the theory of Gr\"obner bases. 

A graph $G$ is said to be {\em simple} if $G$ has no loops and no multiple edges.
A graph  $G$ is said to be {\em bipartite} if the vertex set $V(G)$ of $G$
can be divided into $V(G)=U \sqcup V$ with $E(G) \subset U \times V$,
where $E(G)$ is the edge set of $G$. 
Let $G$ be a finite simple bipartite graph on the vertex set $V(G)=\{u_1,\ldots,u_m\} \sqcup \{v_1,\ldots,v_n\}$ with the edge set $E(G)$. 
Let $R=K[s_1, \dots, s_m, t_1,\dots,t_n]$ and $S=K[x_{ij} : \{u_i,v_j\} \in E(G)]$ be the polynomial rings
over $K$. 
Then the {\em toric ideal} $I_G$ of $G$ is the kernel of the ring homomorphism $\pi : S \rightarrow R, \;\; x_{ij} \mapsto s_i t_j.$ 
An \textit{even cycle} in the bipartite graph $G$ of length $2q$ is a finite sequence of vertices of the form 
\begin{equation}
\label{evencycle}
C =(u_{i_1}, v_{j_1},u_{i_2},v_{j_2},\dots, u_{i_q}, v_{j_q})
\end{equation}
where $\{u_{i_1}, v_{j_q}\} \in E(G)$, $\{u_{i_k}, v_{j_k}\} \in E(G)$ for $1 \le k \le q$, 
$\{u_{i_{k+1}}, v_{j_k}\} \in E(G)$ for $1 \le k \le q-1$, and there exist no repeated vertices. 
Given an even cycle $C$ in \eqref{evencycle}, we write $f_C$ for the binomial 
$f_C =\prod_{k=1}^{q} x_{i_k j_k}- x_{i_1 j_q} \prod_{k=1}^{q-1} x_{i_{k+1} j_k}$ belonging to $I_G$.
The following proposition is due to Villarreal \cite[Proposition 3.1]{Vpaper}.

\begin{Proposition}[{\cite[Corollary 5.12]{binomialideals}}]\label{toricbipartite}
Let $G$ be a bipartite graph. Then the reduced Gr\"{o}bner basis of $I_G$ 
with respect to any monomial order 
consists of the binomials of the form $f_C$, 
where $C$ is an even cycle in $G$. In particular, $I_G$ is generated by those binomials.
\end{Proposition}

Let $\Lambda_M$ be the coherent matching field induced by a $2 \times n$ weight matrix 
\[M=
\begin{pmatrix}
0 & 0 & \cdots & 0 & 0\\
w_1 & w_2 & \cdots & w_{n-1} & w_n
\end{pmatrix},
\]
where $w_i \neq w_j$ for any $i\neq j$, let $G_M$ be a bipartite graph 
on the vertex set $\{u_1,\ldots,u_n\} \sqcup \{v_1,\ldots,v_n\}$ with the edge set 
\[
E(G_M) = \{ \{ u_i, v_j \} : w_i > w_j \}.
\]
It is easy to see that the matching field ideal $J_{\Lambda_M}$ coincides with the toric ideal $I_{G_M}$ of the bipartite graph $G_M$. 
Let 
\[M_{\rm diag}=
\begin{pmatrix}
0 & 0 & \cdots & 0 & 0\\
n & n-1 & \cdots & 2 &1
\end{pmatrix}.
\]
Then $M_{\rm diag}$ is a weight matrix of a diagonal matching field.
It is known \cite[Proposition 1.11]{SZ} that $G_M$ and $G_{M_{\rm diag}}$ are isomorphic as graphs.
Moreover, $J_{M_{\rm diag}}$ has a quadratic Gr\"obner basis with respect to a reverse lexicographic order (\cite[Remark 11.11]{Stu}). 
Thus we have the following in general. 
\begin{Proposition}
\label{HibiRing}
Let $\Lambda$ be a coherent $2 \times n$ matching field
induced by a weight matrix $M$.
Then there exists a reverse lexicographic order $<$ such that
\[
{\mathcal G}=
\{ x_{i \ell} x_{k j} -  x_{i j} x_{k \ell} : 
\{u_i, v_\ell\}, \{u_i,v_j\} , \{u_k,v_j\},\{u_k,v_\ell \} \in E(G_M),
i <k \mbox{ and } j < \ell
 \}
\]
is a (quadratic) Gr\"{o}bner basis of $J_\Lambda$ with respect to $<$.
In particular, $J_\Lambda$ is generated by quadratic binomials in ${\mathcal G}$.
\end{Proposition}

Let $\ab = (a_1,\dots, a_s) \in \ZZ_{>0}^s$ such that $\sum_{i=1}^s a_i = n$,
and let  $\wb_\ab = (w_1,\ldots, w_n)$ be the second row of $M_\ab$.
Let $G_\ab=G_{M_\ab}$. In order to prove our main theorem, we need a quadratic Gr\"obner basis of $I_{G_\ab}$ 
with respect to a reverse lexicographic order defined as follows. 
Let $<$ be a reverse lexicographic order on $S=K[x_{ij} : \{u_i,v_j\} \in E(G_\ab)]$ induced by the ordering
of variables such that $x_{i j} > x_{k \ell}$ if either (i) $i <k$ or (ii) $i=k$ and $j < \ell$.

\begin{Proposition}\label{revlexbipartite}
Let $\ab = (a_1,\dots, a_s) \in \ZZ_{>0}^s$ such that $\sum_{i=1}^s a_i = n$ and $a_i \in \{1,2\}$ for $i \notin \{1,s\}$. 
Then the reduced Gr\"obner basis of $I_{G_\ab}$ with respect to the reverse lexicographic order $<$ is 
\[
{\mathcal G}=
\{ x_{i \ell} x_{k j} -  x_{i j} x_{k \ell} : 
\{u_i, v_\ell\}, \{u_i,v_j\} , \{u_k,v_j\},\{u_k,v_\ell \} \in E(G_\ab), i <k \mbox{ and } j < \ell\}. 
\]
\end{Proposition}
\begin{proof}
Let $G=G_\ab$. From Proposition~\ref{HibiRing}, $I_G$ is generated by ${\mathcal G}$. 
Note that the initial monomial of each $  x_{i \ell} x_{k j} -  x_{i j} x_{k \ell}$ is $x_{i \ell} x_{k j}$
with respect to $<$. 
Applying Buchberger's criterion \cite[Theorem 1.29]{binomialideals}, it is enough to show that 
the $S$-polynomial $S(f,g)$ of any two distinct binomials $f$ and $g$ in ${\mathcal G}$ reduces to 0 with respect to ${\mathcal G}$.
Note that the initial monomials of $f$ and $g$ are different.
From \cite[Lemma 1.27]{binomialideals}, if the initial monomials of $f$ and $g$ are relatively prime, then $S(f,g)$ reduces to $0$. 
Suppose that the initial monomials of $f$ and $g$ have exactly one common variable. 
Then the degree of $S(f,g)$ is three. 
Suppose that the remainder $h$ of $S(f,g)$ with respect to ${\mathcal G}$
is not zero.
By Proposition~\ref{toricbipartite}, we may assume that $h$ is of the form
\[
h=f_C=x_{i_1 j_1} x_{i_2 j_2} x_{i_3 j_3} - x_{i_1 j_3} x_{i_2 j_1} x_{i_3 j_{2}} \neq 0, 
\]
where $C=(i_1, j_1,i_2,j_2, i_3 ,j_3)$ is an even cycle of $G$ of length $6$. 
Let $A =(a_{ij}) $ be the $n\times n$ matrix where 
\[a_{ij}=\left\{
\begin{array}{cc}
1 & \mbox{if } \{u_i,v_j\} \in E(G),\\
0 &  \mbox{otherwise.}
\end{array}\right.\]
Then the cycle $C$ appears in $A$ as one of the following submatrices of $A$: 
\begin{equation}
\label{6mat}
\begin{pmatrix}
 & 1 & 1\\
1 & & 1\\
 1& 1&  
\end{pmatrix},
\begin{pmatrix}
1 & & 1\\
  &1 & 1\\
 1& 1&  
\end{pmatrix},
\begin{pmatrix}
 & 1 & 1\\
1 & 1&  \\
 1& &1 
\end{pmatrix},
\begin{pmatrix}
 1&   & 1\\
1 & 1& \\
 & 1& 1
\end{pmatrix},
\begin{pmatrix}
 1& 1 &  \\
  & 1& 1\\
 1& & 1
\end{pmatrix},
\begin{pmatrix}
 1& 1 &  \\
1 & & 1\\
  & 1& 1
\end{pmatrix}.
\end{equation}
If $B_1 = \begin{pmatrix} 1 & 0 \\ 0 & 1 \end{pmatrix}$ is a submatrix of $A$ 
corresponding to the $k_1$, $k_2$-th rows and the $k_3$, $k_4$-th columns of $A$, then we have 
$w_{k_1} > w_{k_3}$, $w_{k_2} > w_{k_4}$, $w_{k_1} \le w_{k_4}$ and $w_{k_2} \le w_{k_3}$, a contradiction. 
Hence $B_1$ is not a submatrix of $A$.
Similarly, $B_2=\begin{pmatrix} 0 & 1 \\ 1 & 0 \end{pmatrix}$ is not a submatrix of $A$. 
Thus it follows that, each matrix in (\ref{6mat}) contains at most one $0$,
and hence $C$ has at least two chords.
If $C$ has a chord at a position marked by $*$ in 
\[
\begin{pmatrix}
* & 1 & 1\\
1 & & 1\\
 1& 1& *
\end{pmatrix},
\begin{pmatrix}
1 &* & 1\\
* &1 & 1\\
 1& 1& *
\end{pmatrix},
\begin{pmatrix}
* & 1 & 1\\
1 & 1& * \\
 1& *&1 
\end{pmatrix},
\begin{pmatrix}
 1& *  & 1\\
1 & 1& *\\
 & 1& 1
\end{pmatrix},
\begin{pmatrix}
 1& 1 &  \\
 * & 1& 1\\
 1& * & 1
\end{pmatrix},
\begin{pmatrix}
 1& 1 &  \\
1 & * & 1\\
  & 1& 1
\end{pmatrix},
\]
then $C$ has a submatrix of one of $B=\begin{pmatrix} * & 1 \\ 1 & 1 \end{pmatrix}$ and $\begin{pmatrix} 1 & 1 \\ 1 & * \end{pmatrix}$.
Then $B$ corresponds to a cycle $C'$ of length 4 in $G$ and $f_{C'}$ belongs to ${\mathcal G}$. 
Moreover, the initial monomial of $f_{C'}$ divides one of the monomials of $h$. 
This contradicts the hypothesis that $h$ is a remainder with respect to ${\mathcal G}$. Thus we may assume that 
\begin{align*}
&\quad\;\begin{matrix}
\ m_1 & \ m_2 & \  m_3 
\end{matrix}\\
\begin{matrix}
\ell_1\\
\ell_2\\
\ell_3
\end{matrix}
&\begin{pmatrix}
\ \ 1 \  \  & \ \ 1 \  \ & \ \ 1 \  \  \\
1 & 0 & 1\\
 1 & 1& 1
\end{pmatrix},
\end{align*}
where $\ell_1 < \ell_2 < \ell_3$ and $m_1 < m_2 < m_3$
is a submatrix of $A$.
Then
\begin{eqnarray*}
\min\{w_{\ell_1},w_{\ell_3}\}  >  \max\{w_{m_1} , w_{m_2} , w_{m_3}\} \;\; \text{ and }\;\;
w_{m_2} \geq w_{\ell_2} > \max\{w_{m_1}  , w_{m_3}\}. 
\end{eqnarray*}
Hence we have
\[
w_{\ell_1}, w_{\ell_3} >w_{m_2} \ge w_{\ell_2} > w_{m_1}  , w_{m_3}.
\]
If $w_{\ell_1} > w_{\ell_3}$ (resp. $w_{m_1} > w_{m_3}$), 
then $w_{\ell_2} > w_{\ell_3}$ (resp. $w_{m_1} > w_{m_2}$) by condition \eqref{b_d_condition}.
This is a contradiction. Thus we have 
\[
w_{\ell_3} > w_{\ell_1} >w_{m_2} \ge w_{\ell_2} > w_{m_3}  > w_{m_1}.
\]

We now show that $m_1 < \ell_1 < \ell_2 < m_3 < \ell_3$.
Suppose that $(\ell_2 <) \ell_3 < m_3$. 
Since $w_{\ell_2} > w_{m_3}$, we have $w_{\ell_2} > w_{\ell_3}$ by condition \eqref{b_d_condition}, a contradiction. 
Suppose that $\ell_1 < m_3 < \ell_2$. 
Since $w_{\ell_1} > w_{\ell_2}$, we have $w_{m_3} > w_{\ell_2}$ by condition \eqref{b_d_condition}, a contradiction. 
Suppose that $$(m_2 <)\  m_3 < \ell_1 (< \ell_2).$$ 
Since $w_{m_2} > w_{\ell_2}$, we have $w_{m_3} > w_{\ell_1}$ by condition \eqref{b_d_condition}, a contradiction.
Thus $\ell_2 < m_3 < \ell_3$. 
Suppose that $\ell_2 < m_1 ( < m_3 )$. 
Since $w_{\ell_2} > w_{m_3}$, we have $w_{m_1} > w_{m_3}$, a contradiction.
Thus $m_1 < \ell_2 < m_3 < \ell_3$.
Suppose that $\ell_1 < m_1 < \ell_2$. 
Since $w_{\ell_1} > w_{\ell_2}$, we have $w_{m_1} > w_{\ell_2}$, a contradiction.
Therefore $m_1 < \ell_1 < \ell_2 < m_3 < \ell_3$.

Since $w_{m_1} < w_{\ell_1} > w_{\ell_2} > w_{m_3} < w_{\ell_3}$, this contradicts to condition \eqref{b_d_condition2}.
\end{proof}

\medskip


\section{Quadratic Gr\"{o}bner bases of $s$-block diagonal matching fields}\label{sec:qGB}

Recall that the matching field ideal $J_\Lambda$ of a matching field $\Lambda$ is the kernel of a ring homomorphism 
$$\psi_\Lambda : S=K[P_I : I \in  \Ib_{r,n}] \rightarrow R=K[x_{ij} : 1 \le i \le r, 1 \le j \le n]$$
defined by $\psi_\Lambda(P_I) = \sgn(\sigma) x_{\sigma(1) i_1} \cdots x_{\sigma(r) i_r} $, 
where $I=\{i_1,\ldots,i_r\} \in \Ib_{r,n}$ with $i_1 < \cdots < i_r$ and $\sigma = \Lambda(I)$. 
From now on, we identify a variable $P_I$ with 
\[
\begin{bmatrix}
i_1\\
i_2\\
i_3\\
\vdots\\
i_r
\end{bmatrix}
\mbox{ if } \Lambda(I) = {\rm id} \hspace{1cm}
\mbox{ and } \hspace{1cm}
\begin{bmatrix}
i_2\\
i_1\\
i_3\\
\vdots\\
i_r
\end{bmatrix}
\mbox{ if } \Lambda(I) =  (1 \ 2).
\]

\begin{Lemma}\label{binomialslemma}
Let $\Lambda_\ab$ be an $s$-block diagonal matching field of size $r \times n$ associated with $\ab \in \ZZ_{>0}^s$. Let 
\begin{eqnarray*}
&1 \le i_1 < i_2 < \cdots < i_r \le n,\\
&1 \le j_1 < j_2 < \cdots < j_r \le n,\\
&i_k' = \min \{ i_k, j_k\} \;\mbox{ for }\; k=1,\ldots,r,\\
&j_k' = \max \{ i_k, j_k\} \;\mbox{ for }\; k=1,\ldots,r.
\end{eqnarray*}
Then we have the following:
\begin{align}
\begin{bmatrix}
i_1\\
i_2\\
i_3\\
\vdots\\
i_r
\end{bmatrix},
\begin{bmatrix}
j_1\\
j_2\\
j_3\\
\vdots\\
j_r
\end{bmatrix} \in S
\mbox{ and } 
\begin{matrix}
i_1 < j_1
\end{matrix}
\ \Longrightarrow \ 
\begin{bmatrix}
i_1\\
i_2\\
i_3\\
\vdots\\
i_r
\end{bmatrix}
\begin{bmatrix}
j_1\\
j_2\\
j_3\\
\vdots\\
j_r
\end{bmatrix}
-
\begin{bmatrix}
i_1\\
i_2'\\
i_3'\\
\vdots\\
i_r'
\end{bmatrix}
\begin{bmatrix}
j_1\\
j_2'\\
j_3'\\
\vdots\\
j_r'
\end{bmatrix} \in J_{\Lambda_\ab}, \tag{i}
\end{align}
\begin{align}
\begin{bmatrix}
i_2\\
i_1\\
i_3\\
\vdots\\
i_r
\end{bmatrix},
\begin{bmatrix}
j_1\\
j_2\\
j_3\\
\vdots\\
j_r
\end{bmatrix}\in S
\mbox{ and } 
\begin{matrix}
i_2 < j_1
\end{matrix}
\ \Longrightarrow \ 
\begin{bmatrix}
i_2\\
i_1\\
i_3\\
\vdots\\
i_r
\end{bmatrix}
\begin{bmatrix}
j_1\\
j_2\\
j_3\\
\vdots\\
j_r
\end{bmatrix}
-
\begin{bmatrix}
i_2\\
i_1\\
i_3'\\
\vdots\\
i_r'
\end{bmatrix}
\begin{bmatrix}
j_1\\
j_2\\
j_3'\\
\vdots\\
j_r'
\end{bmatrix} \in J_{\Lambda_\ab},\tag{ii} 
\end{align}
\begin{align}
\begin{bmatrix}
i_1\\
i_2\\
i_3\\
\vdots\\
i_r
\end{bmatrix},
\begin{bmatrix}
j_2\\
j_1\\
j_3\\
\vdots\\
j_r
\end{bmatrix}\in S
\mbox{ and } 
\begin{matrix}
i_1 < j_2\\
i_2 < j_3
\end{matrix}
\ \Longrightarrow \ 
\begin{bmatrix}
i_1\\
i_2\\
i_3\\
\vdots\\
i_r
\end{bmatrix}
\begin{bmatrix}
j_2\\
j_1\\
j_3\\
\vdots\\
j_r
\end{bmatrix}
-
\begin{bmatrix}
i_1\\
i_2\\
i_3'\\
\vdots\\
i_r'
\end{bmatrix}
\begin{bmatrix}
j_2\\
j_1\\
j_3'\\
\vdots\\
j_r'
\end{bmatrix} \in J_{\Lambda_\ab},\tag{iii} 
\end{align}
\begin{align}
\begin{bmatrix}
i_1\\
i_2\\
i_3\\
\vdots\\
i_r
\end{bmatrix},
\begin{bmatrix}
j_2\\
j_1\\
j_3\\
\vdots\\
j_r
\end{bmatrix}\in S
\mbox{ and } 
\begin{matrix}
i_1 < j_2\\
j_3 \le i_2
\end{matrix}
\ \Longrightarrow \ 
\begin{bmatrix}
i_1\\
i_2\\
i_3\\
\vdots\\
i_r
\end{bmatrix}
\begin{bmatrix}
j_2\\
j_1\\
j_3\\
\vdots\\
j_r
\end{bmatrix}
-
\begin{bmatrix}
i_1\\
j_1\\
i_3'\\
\vdots\\
i_r'
\end{bmatrix}
\begin{bmatrix}
j_2\\
i_2\\
j_3'\\
\vdots\\
j_r'
\end{bmatrix} \in J_{\Lambda_\ab},\tag{iv} 
\end{align}
\begin{align}
\begin{bmatrix}
i_2\\
i_1\\
i_3\\
\vdots\\
i_r
\end{bmatrix},
\begin{bmatrix}
j_2\\
j_1\\
j_3\\
\vdots\\
j_r
\end{bmatrix} \in S
\mbox{ and } 
\begin{matrix}
i_2 < j_2
\end{matrix}
\ \Longrightarrow \ 
\begin{bmatrix}
i_2\\
i_1\\
i_3\\
\vdots\\
i_r
\end{bmatrix}
\begin{bmatrix}
j_2\\
j_1\\
j_3\\
\vdots\\
j_r
\end{bmatrix} 
-
\begin{bmatrix}
i_2\\
i_1\\
i_3'\\
\vdots\\
i_r'
\end{bmatrix}
\begin{bmatrix}
j_2\\
j_1\\
j_3'\\
\vdots\\
j_r'
\end{bmatrix}  \in J_{\Lambda_\ab},\tag{v} 
\end{align}
\begin{align}
\begin{bmatrix}
\ell_1\\
\ell_2\\
\ell_3\\
\vdots\\
\ell_r
\end{bmatrix},
\begin{bmatrix}
\ell_1\\
m_2\\
m_3\\
\vdots\\
m_r
\end{bmatrix}\in S
\mbox{ and } \ \ \ 
\begin{matrix}
\ell_k' = \min \{ \ell_k, m_k \}\\
m_k' = \max \{ \ell_k, m_k\}\\
\\
 \mbox{ for } k =2,3,\dots,r
\end{matrix}
\Longrightarrow \ 
\begin{bmatrix}
\ell_1\\
\ell_2\\
\ell_3\\
\vdots\\
\ell_r
\end{bmatrix}
\begin{bmatrix}
\ell_1\\
m_2\\
m_3\\
\vdots\\
m_r
\end{bmatrix}
-
\begin{bmatrix}
\ell_1\\
\ell_2'\\
\ell_3'\\
\vdots\\
\ell_r'
\end{bmatrix}
\begin{bmatrix}
\ell_1\\
m_2'\\
m_3'\\
\vdots\\
m_r'
\end{bmatrix} \in J_{\Lambda_\ab}, \tag{vi} 
\end{align}
\begin{align}
\begin{bmatrix}
\ell_1\\
\ell_2\\
\ell_3\\
\vdots\\
\ell_r
\end{bmatrix},
\begin{bmatrix}
m_1\\
m_2\\
m_3\\
\vdots\\
m_r
\end{bmatrix}\in S
\mbox{ and } \ \ \ 
\begin{matrix}
w_{\ell_1} > w_{m_2},\\
w_{m_1} > w_{\ell_2}\\
\ell_1 < m_1, \\
m_2 < \ell_2,\\
\ell_k \le m_k \mbox{ for } k \ge 3
\end{matrix}
\Longrightarrow \ 
\begin{bmatrix}
\ell_1\\
\ell_2\\
\ell_3\\
\vdots\\
\ell_r
\end{bmatrix}
\begin{bmatrix}
m_1\\
m_2\\
m_3\\
\vdots\\
m_r
\end{bmatrix}
-
\begin{bmatrix}
\ell_1\\
m_2\\
\ell_3\\
\vdots\\
\ell_r
\end{bmatrix}
\begin{bmatrix}
m_1\\
\ell_2\\
m_3\\
\vdots\\
m_r
\end{bmatrix} \in J_{\Lambda_\ab}. \tag{vii} 
\end{align}
\end{Lemma}
\begin{proof}
It is easy to see that, if a binomial $f$ appearing in (i) -- (vii) belongs to the polynomial ring $S$, 
then $\psi_\Lambda(f) = 0$ and hence $f \in J_{\Lambda_\ab}$. 
Thus, it is enough to show that the second monomial of any binomial appearing in (i) -- (vii) belongs to the polynomial ring $S$. 
First, we show that 
\begin{eqnarray}\label{daisyo}
i_k' < i_{k+1}' \mbox{ and } j_k' < j_{k+1}' \mbox{ for } k=1,2,\dots , r-1.
\end{eqnarray}
Suppose that $(i_k', j_k') = (i_k, j_k)$. Since $i_k' = i_k \le j_k$, we have $i_k' < i_{k+1}, j_{k+1}$. Hence $i_k' < i_{k+1}'$.
In addition, $j_k' = j_k  < j_{k+1} \le j_{k+1}'$. 
Suppose that $(i_k', j_k') = (j_k, i_k)$. Since $i_k' = j_k \le i_k$, we have $i_k' < i_{k+1}, j_{k+1}$. 
Hence $i_k' < i_{k+1}'$. In addition, $j_k' = i_k  < i_{k+1} \le j_{k+1}'$. 

\medskip

\noindent
(i) From \eqref{daisyo}, $i_1 = i_1' < i_2' < \cdots < i_r'$ and $j_1 =j_1' < j_2' < \cdots < j_r'$.
By the hypothesis, $w_{i_1} > w_{i_2}$ and $w_{j_1} > w_{j_2}$ 
If $(i_2', j_2') = (i_2, j_2)$, then  $w_{i_1} > w_{i_2'}$ and $w_{j_1} > w_{j_2'}$ is trivial.
Suppose that $(i_2', j_2') = (j_2, i_2)$. Then $i_1 < j_1 < j_2 = i_2' \le j_2' = i_2$. Hence $w_{i_1} > w_{i_2'}$ and $w_{j_1} > w_{j_2'}$.

\medskip

\noindent
(ii) Since $i_2< j_1 < j_2$, we have $(i_2', j_2') = (i_2, j_2)$.
From \eqref{daisyo}, $i_1  < i_2 =i_2' < i_3' < \cdots < i_r'$ and $j_1 < j_2 =j_2' < \cdots < j_r'$. 
By the hypothesis, $w_{i_1} < w_{i_2}$, $w_{j_1} > w_{j_2}$.

\medskip

\noindent
(iii) From \eqref{daisyo}, $j_1 < j_2 (< j_3) \le j_3' < \cdots < j_r'$ and $i_3' < \cdots < i_r'$. 
Since $i_2 < j_3$, we have $i_1< i_2 < i_3'$. By the hypothesis, $w_{i_1} > w_{i_2}$, $w_{j_1} < w_{j_2}$,

\medskip

\noindent
(iv) From \eqref{daisyo}, $i_3' < \cdots < i_r'$ and $j_3' < \cdots < j_r'$. 
By the hypothesis, $w_{i_1} > w_{i_2}$, $w_{j_1} < w_{j_2}$, and $i_1 < j_2 < j_3 \le i_2 < i_3$. 
If $i_1 \le j_1 (< j_2 < i_2)$, then $w_{j_1} > w_{j_2}$, a contradiction. 
Hence $$j_1 < i_1 (< j_2 < j_3 \le i_2 < i_3).$$ 
Thus $j_1 < i_1 < i_3'$ and $j_2 < i_2 < i_3 = j_3'$. Since $w_{i_1} > w_{i_2}$, we have $w_{i_1} > w_{j_2}$. 
Suppose that $w_{i_1} \le w_{j_1}$. It then follows that $w_{i_1} \le w_{j_1} <  w_{j_2}$, a contradiction. Hence $w_{i_1} > w_{j_1}$. 

\medskip

\noindent
(v) From \eqref{daisyo}, $i_3' < \cdots < i_r'$ and $j_3' < \cdots < j_r'$.  
By the hypothesis, $w_{i_1} < w_{i_2}$, $w_{j_1} < w_{j_2}$, and $i_2 < j_2< j_3$. Then $i_1< i_2 <  i_3'$ and $j_1 < j_2 < j_3 \le j_3'$.

\medskip

\noindent
(vi) By the hypothesis, $w_{\ell_1} > w_{\ell_2'}$ and $w_{\ell_1} > w_{m_2'}$ is trivial.
By the same argument in the proof of \eqref{daisyo}, 
it follows that $\ell_1 , \ell_2' < \ell_3' < \cdots < \ell_r'$ and $\ell_1 , m_2' < m_3' < \cdots < m_r'$.  

\medskip

\noindent
(vii) By the hypothesis, $w_{\ell_1} > w_{m_2}$ and $w_{m_1} > w_{\ell_2}$. Moreover, since 
$$\begin{bmatrix}
\ell_1\\
\ell_2\\
\ell_3\\
\vdots\\
\ell_r
\end{bmatrix}
\mbox{ and }
\begin{bmatrix}
m_1\\
m_2\\
m_3\\
\vdots\\
m_r
\end{bmatrix}$$
belong to $S$, $\ell_2 < \ell_3$ and $m_2 < m_3$. Hence $m_2 < \ell_2 < \ell_3$ and $\ell_2 < \ell_3 \le m_3$.

\medskip

Therefore, the second monomial of any binomial appearing in (i) -- (vii) belongs to the polynomial ring $S$. 
\end{proof}

Let $<$ be a reverse lexicographic order induced by the ordering of variables such that 
\begin{equation}
\label{revlex_rule}
\begin{bmatrix}
i_1\\
i_2\\
\vdots\\
i_r
\end{bmatrix}
>
\begin{bmatrix}
j_1\\
j_2\\
\vdots\\
j_r
\end{bmatrix}
\Longleftrightarrow
i_k = j_k \mbox{ for } k=1,2,\dots,t-1 \mbox{ and } i_{t} < j_{t}.
\end{equation}
The following theorem is the first main result of the present paper: 
\begin{Theorem}\label{thm:quad}
Let $\Lambda_\ab$ be an $s$-block diagonal matching field of size $r \times n$ associated with $\ab = (a_1,\dots, a_s) \in \ZZ_{>0}^s$ 
such that $\sum_{i=1}^s a_i = n$ and $a_i \in \{1,2\}$ for $i \notin \{1,s\}$. 
Let ${\mathcal G}_\ab$ be the set of all binomials appearing in Lemma {\rm \ref{binomialslemma}} {\rm (i) -- (vii)}.
Then ${\mathcal G}_\ab$ is a quadratic Gr\"{o}bner basis of $J_{\Lambda_\ab}$ with respect to the reverse lexicographic order $<$.
\end{Theorem}
\begin{proof}
Suppose that ${\mathcal G}_\ab$ is not a Gr\"{o}bner basis of $J_{\Lambda_\ab}$.
By \cite[Theorem 3.11]{binomialideals}, there exists an irreducible homogeneous binomial $u-v \in J_{\Lambda_\ab}$
such that neither $u$ nor $v$ belongs to the monomial ideal generated by
the initial monomials of the binomials in ${\mathcal G}_\ab$.
Note that the initial monomial of each binomial in ${\mathcal G}_\ab$
is the first monomial in Lemma \ref{binomialslemma} (if the polynomial is not zero).
Let
\[
u=
\begin{bmatrix}
i_{11}\\
i_{12}\\
\vdots\\
i_{1r}
\end{bmatrix}
\begin{bmatrix}
i_{21}\\
i_{22}\\
\vdots\\
i_{2r}
\end{bmatrix}
\cdots
\begin{bmatrix}
i_{d1}\\
i_{d2}\\
\vdots\\
i_{dr}
\end{bmatrix}
\ \text{ and }\ 
v=
\begin{bmatrix}
j_{11}\\
j_{12}\\
\vdots\\
j_{1r}
\end{bmatrix}
\begin{bmatrix}
j_{21}\\
j_{22}\\
\vdots\\
j_{2r}
\end{bmatrix}
\cdots
\begin{bmatrix}
j_{d1}\\
j_{d2}\\
\vdots\\
j_{dr}
\end{bmatrix},
\]
where 
\[
\begin{bmatrix}
i_{11}\\
i_{12}\\
\vdots\\
i_{1r}
\end{bmatrix}
>
\begin{bmatrix}
i_{21}\\
i_{22}\\
\vdots\\
i_{2r}
\end{bmatrix}
> \cdots >
\begin{bmatrix}
i_{d1}\\
i_{d2}\\
\vdots\\
i_{dr}
\end{bmatrix}
\mbox{ and }
\begin{bmatrix}
j_{11}\\
j_{12}\\
\vdots\\
j_{1r}
\end{bmatrix}
>
\begin{bmatrix}
j_{21}\\
j_{22}\\
\vdots\\
j_{2r}
\end{bmatrix}
> \cdots >
\begin{bmatrix}
j_{d1}\\
j_{d2}\\
\vdots\\
j_{dr}
\end{bmatrix}.
\]
If $i_{\mu k} > i_{\eta k}$ for some $3 \le k \le r$ and $1 \le \mu < \eta \le d$, then 
$u$ is divisible by the initial monomial of a binomial appearing in Lemma \ref{binomialslemma} (i) -- (vi). Thus we may assume that
\[ i_{1 k} \le \cdots \le i_{d k} \mbox{ and } j_{1 k} \le \cdots \le j_{d k} \] for all $k \neq 2$.
Since $u-v$ belongs to $J_{\Lambda_\ab}$, it follows that $i_{\mu k } = j_{\mu k }$ for any $1 \le \mu \le d$ and $k\neq 2$.
If $i_{\mu 2 } = j_{\mu 2 }$  for any $1 \le \mu \le d$, then $u-v=0$, a contradiction. Hence
\[
\begin{bmatrix}
i_{11}\\
i_{12}
\end{bmatrix}
\begin{bmatrix}
i_{21}\\
i_{22}
\end{bmatrix}
\cdots
\begin{bmatrix}
i_{d1}\\
i_{d2}
\end{bmatrix}
-
\begin{bmatrix}
i_{11}\\
j_{12}
\end{bmatrix}
\begin{bmatrix}
i_{21}\\
j_{22}
\end{bmatrix}
\cdots
\begin{bmatrix}
i_{d1}\\
j_{d2}
\end{bmatrix}
\]
is a nonzero binomial in $I_{G_\ab}$. We may assume that $u$ is the initial monomial of $u-v$. 
From Proposition \ref{revlexbipartite}, there exists a binomial 
\[
\begin{bmatrix}
i_{\mu 1}\\
i_{\mu 2}
\end{bmatrix}
\begin{bmatrix}
i_{\eta 1}\\
i_{\eta 2}
\end{bmatrix}
-
\begin{bmatrix}
i_{\mu 1}\\
i_{\eta 2}
\end{bmatrix}
\begin{bmatrix}
i_{\eta 1}\\
i_{\mu 2}
\end{bmatrix}
\in I_{G_\ab},
\]
where $i_{\mu 1} < i_{\eta 1}$ and $i_{\mu 2} > i_{\eta 2}$. Then 
$\begin{bmatrix}
i_{\mu 1}\\
i_{\mu 2}\\
\vdots\\
i_{\mu r}
\end{bmatrix}
\begin{bmatrix}
i_{\eta 1}\\
i_{\eta 2}\\
\vdots\\
i_{\eta  r}
\end{bmatrix}$
satisfies the condition in Lemma~\ref{binomialslemma} (vii), and hence 
$f=
\begin{bmatrix}
i_{\mu 1}\\
i_{\mu 2}\\
i_{\mu 3}\\
\vdots\\
i_{\mu r}
\end{bmatrix}
\begin{bmatrix}
i_{\eta 1}\\
i_{\eta 2}\\
i_{\eta 3}\\
\vdots\\
i_{\eta  r}
\end{bmatrix}
-
\begin{bmatrix}
i_{\mu 1}\\
i_{\eta 2}\\
i_{\mu 3}\\
\vdots\\
i_{\mu r}
\end{bmatrix}
\begin{bmatrix}
i_{\eta 1}\\
i_{\mu 2}\\
i_{\eta 3}\\
\vdots\\
i_{\eta  r}
\end{bmatrix}$
is a binomial belonging to ${\mathcal G}_\ab$ whose initial monomial is the first monomial.
Since the first monomial of $f$ divides $u$, this is a contradiction. 
\end{proof}

\medskip


\section{Toric degenerations associated to $s$-block diagonal matching fields}\label{sec:toric_deg}

In this section, we provide a new family of toric degenerations of $\Gr(r,n)$ (Corollary~\ref{cor:toric_deg}). 
For this goal, we recall the notion of SAGBI bases for the Pl\"ucker algebra.

\begin{Definition}[SAGBI bases for the Pl\"ucker algebra]
We say that the generating set $\{\det(\xb_I) : I \in \Ib_{r,n}\} \subset R$ of the Pl\"ucker algebra $\Ac_{r,n}$ 
is a {\em SAGBI basis with respect to a weight matrix $M \in \RR^{r \times n}$} if $\ini_M(\Ac_{r,n})=K[\ini_M(\det(\xb_I)) : I \in \Ib_{r,n}]$, 
where $\ini_M(\Ac_{r,n})=K[\ini_M(f) : f \in \Ac_{r,n}]$. 
\end{Definition}

\begin{Theorem}[{\cite[Theorem 11.4]{Stu}}]\label{thm:enough}
The generating set $\{\det(\xb_I) : I \in \Ib_{r,n}\}$ of the Pl\"ucker algebra $\Ac_{r,n}$ 
is a SAGBI basis with respect to a weight matrix $M \in \RR^{r \times n}$ 
if and only if $J_\Lambda=\ini_{w_M}(I_{r,n})$ holds, where $w_M$ is the weight vector on $S$ induced by $M$ 
and $\Lambda$ is the coherent matching field induced by $M$. 
\end{Theorem}

The following is the second main result of the present paper and 
the essential statement for the existence of toric degenerations arising from $s$-block diagonal matching fields with certain conditions. 
\begin{Theorem}\label{thm:toric_deg}
Let $\ab = (a_1,\dots, a_s) \in \ZZ_{>0}^s$ such that $\sum_{i=1}^s a_i = n$ and $a_i \in \{1,2\}$ for $i \notin \{1,s\}$. 
Then the generating set $\{\det(\xb_I) : I \in \Ib_{r,n}\}$ forms a SAGBI basis for the Pl\"ucker algebra $\Ac_{r,n}$ 
with respect to the weight matrix $M_\ab$ associated to the block diagonal matching field $\Lambda_\ab$. 
\end{Theorem}
Note that the case $s=2$ of Theorem~\ref{thm:toric_deg} was proved in \cite[Theorem 4.3]{CM}. 
Moreover, as described below, the proof of Theorem~\ref{thm:toric_deg} looks simpler than that of \cite[Theorem 4.3]{CM}. 

Let $\Lambda_\ab$ be an $s$-block diagonal matching field associated with $\ab = (a_1,\dots, a_s) \in \ZZ_{>0}^s$ such that $\sum_{i=1}^s a_i = n$. 
Let $A_\ab$ be the toric ring arising from $\Lambda_\ab$, i.e., $S/J_{\Lambda_\ab}$, 
and let $A_0$ denote that of the diagonal matching field, i.e., $A_0=A_{(n)}$. 
Let $[A]_2$ denote the $K$-vector subspace of $A$ generated by elements of degree 2 in $A$. 

Our proof of Theorem~\ref{thm:toric_deg} essentially consists of Theorem~\ref{thm:quad} and the following lemma.
Note that the condition ``$a_i \in \{1,2\}$ for $i \notin \{1,s\}$'' is not required in this lemma.

\begin{Lemma}\label{standard_monomials}
Let $\ab = (a_1,\dots, a_s) \in \ZZ_{>0}^s$ such that $\sum_{i=1}^s a_i = n$.
Then $\dim_K ( [A_\ab]_2) = \dim_K ([A_0]_2)$. 
\end{Lemma}

\begin{proof}
Let ${\mathcal G}_\ab$ be the set of all binomials appearing in Lemma {\rm \ref{binomialslemma}} {\rm (i) -- (vii)}
and let $<$ be the reverse lexicographic order defined by (\ref{revlex_rule}).
Since we do not assume $a_i \in \{1,2\}$ for $i \notin \{1,s\}$, 
${\mathcal G}_\ab$ is not necessarily a Gr\"{o}bner basis with respect to $<$.
Let $\langle \ini_< (\Gc_\ab) \rangle$ be the monomial ideal generated by the initial monomials of the binomials in ${\mathcal G}_\ab$. 
Then $\langle \ini_< (\Gc_\ab) \rangle$ is a subideal of the initial ideal 
$\ini_<( J_{\Lambda_\ab} )$ of $J_{\Lambda_\ab}$.
Let $\Mc$ be the set of all quadratic monomials in $S \setminus \langle \ini_< (\Gc_\ab) \rangle$.
It is known that the set of all quadratic monomials in $S \setminus \ini_<( J_{\Lambda_\ab} )$ is a basis of $ [A_\ab]_2$ 
(see, e.g., \cite[Proposition 1.1]{Stu}). Since $ \langle \ini_< (\Gc_\ab) \rangle \subset \ini_<( J_{\Lambda_\ab} )$, 
it follows that $\Mc$ spans $ [A_\ab]_2$. Thus it is enough to show that
\begin{itemize}
\item[(i)] the cardinality of $\Mc$ is independent of $\ab$; 
\item[(ii)] $\Mc$ is linearly independent over $K$. 
\end{itemize}

First, we prove (i). Let
\begin{equation}
\label{std}
u=
\begin{bmatrix}
\ell_1\\
\ell_2\\
\ell_3\\
\vdots\\
\ell_r
\end{bmatrix}
\begin{bmatrix}
m_1\\
m_2\\
m_3\\
\vdots\\
m_r
\end{bmatrix}
\in \Mc. \tag{$*$}
\end{equation}

\smallskip

\noindent
{\bf Case 1}: $|\{ \ell_1, \ell_2, m_1, m_2\}|=2$. 

Then $\ell_1 = m_1$ and $\ell_2 = m_2$. Given positive integers $1 \le \mu < \nu \le n$, 
the number of monomials $u \in \Mc$ of the form in \eqref{std} such that $\{\ell_1, \ell_2\} = \{\mu, \nu\}$ is equal to 
\[|\{(\ell_3,\dots,\ell_r,m_3,\dots,m_r) : \ell_k < \ell_{k+1}, m_k < m_{k+1}, m_k \le \ell_k \mbox{ for } k \ge 3, \mbox{ and }\nu < \ell_3\}|\]
since such a monomial is of the form \eqref{std},
where $\ell_k < \ell_{k+1}, m_k < m_{k+1}, m_k \le \ell_k$ for $k \ge 3$, and $\nu < \ell_3$.
The number of such monomials is independent of $\ab$.

\smallskip

\noindent
{\bf Case 2}: $|\{ \ell_1, \ell_2, m_1, m_2\}|=3$. 

We show that, given distinct positive integers $1 \le \lambda, \mu_1, \mu_2 \le n$, the number of monomials $u \in \Mc$ of the form \eqref{std}
such that $\{ \ell_1, \ell_2, m_1, m_2\} = \{\lambda, \mu_1,\mu_2\}$ and $\{ \ell_1, \ell_2\} \cap \{m_1, m_2\} = \{\lambda\}$ 
are independent of $\ab$.
If $\lambda < \mu_1< \mu_2$, then the possible patterns are listed in Table~\ref{Case 2-1}, 
where $\mu_1 < \ell_3$, $\mu_2 < m_3$ and $\ell_3 \le m_3$. 
\begin{table}[h]
  \begin{tabular}{|c|c|c|c|}
\hline
$w_\lambda > w_{\mu_1} >  w_{\mu_2}$ & $w_\lambda > w_{\mu_1} < w_{\mu_2}$& $w_\lambda < w_{\mu_1} > w_{\mu_2}$  & $w_\lambda < w_{\mu_1} < w_{\mu_2}$\\
\hline
$\begin{bmatrix}
\lambda\\
\mu_1\\
\ell_3\\
\vdots
\end{bmatrix}
\begin{bmatrix}
\lambda\\
\mu_2\\
m_3\\
\vdots
\end{bmatrix}$
&
$\begin{bmatrix}
\lambda\\
\mu_1\\
\ell_3\\
\vdots
\end{bmatrix}
\begin{bmatrix}
\mu_2\\
\lambda\\
m_3\\
\vdots
\end{bmatrix}$
&
$\begin{bmatrix}
\mu_1\\
\lambda\\
\ell_3\\
\vdots
\end{bmatrix}
\begin{bmatrix}
\mu_2\\
\lambda\\
m_3\\
\vdots
\end{bmatrix}$
&
$\begin{bmatrix}
\mu_1\\
\lambda\\
\ell_3\\
\vdots
\end{bmatrix}
\begin{bmatrix}
\mu_2\\
\lambda\\
m_3\\
\vdots
\end{bmatrix}$\\
\hline
  \end{tabular}
\caption{Case 2-1}
\label{Case 2-1}
\end{table}
If $ \mu_1 <\lambda < \mu_2$, then the possible patterns are listed in Table~\ref{Case 2-2}, 
where $\lambda < \ell_3$, $\mu_2 < m_3$ and $\ell_3 \le m_3$.
\begin{table}[h]
  \begin{tabular}{|c|c|c|c|}
\hline
$ w_{\mu_1} > w_\lambda > w_{\mu_2}$ & $w_{\mu_1} > w_\lambda < w_{\mu_2}$& $w_{\mu_1} < w_\lambda > w_{\mu_2}$  & $w_{\mu_1} < w_\lambda < w_{\mu_2}$\\
\hline
$\begin{bmatrix}
\mu_1\\
\lambda\\
\ell_3\\
\vdots
\end{bmatrix}
\begin{bmatrix}
\lambda\\
\mu_2\\
m_3\\
\vdots
\end{bmatrix}$
&
$\begin{bmatrix}
\mu_1\\
\lambda\\
\ell_3\\
\vdots
\end{bmatrix}
\begin{bmatrix}
\mu_2\\
\lambda\\
m_3\\
\vdots
\end{bmatrix}$
&
$\begin{bmatrix}
\lambda\\
\mu_1\\
\ell_3\\
\vdots
\end{bmatrix}
\begin{bmatrix}
\lambda\\
\mu_2\\
m_3\\
\vdots
\end{bmatrix}$
&
$\begin{bmatrix}
\lambda\\
\mu_1\\
\ell_3\\
\vdots
\end{bmatrix}
\begin{bmatrix}
\mu_2\\
\lambda\\
m_3\\
\vdots
\end{bmatrix}$\\
\hline
  \end{tabular}
\caption{Case 2-2}
\label{Case 2-2}
\end{table}
If $ \mu_1 < \mu_2 <\lambda $, then the possible patterns are listed in Table~\ref{Case 2-3}, where $\lambda < \ell_3 \le m_3$.
\begin{table}[h]
  \begin{tabular}{|c|c|c|c|}
\hline
$ w_{\mu_1} > w_{\mu_2}  > w_\lambda$ & $w_{\mu_1} > w_{\mu_2} < w_\lambda$& $w_{\mu_1} < w_{\mu_2} > w_\lambda$  & $w_{\mu_1} < w_{\mu_2} < w_\lambda$\\
\hline
$\begin{bmatrix}
\mu_1\\
\lambda\\
\ell_3\\
\vdots
\end{bmatrix}
\begin{bmatrix}
\mu_2\\
\lambda\\
m_3\\
\vdots
\end{bmatrix}$
&
$\begin{bmatrix}
\lambda\\
\mu_1\\
\ell_3\\
\vdots
\end{bmatrix}
\begin{bmatrix}
\lambda\\
\mu_2\\
m_3\\
\vdots
\end{bmatrix}$
&
$\begin{bmatrix}
\mu_2\\
\lambda\\
\ell_3\\
\vdots
\end{bmatrix}
\begin{bmatrix}
\lambda\\
\mu_1\\
m_3\\
\vdots
\end{bmatrix}$
&
$\begin{bmatrix}
\lambda\\
\mu_1\\
\ell_3\\
\vdots
\end{bmatrix}
\begin{bmatrix}
\lambda\\
\mu_2\\
m_3\\
\vdots
\end{bmatrix}$\\
\hline
  \end{tabular}
\caption{Case 2-3}
\label{Case 2-3}
\end{table}
Thus the number of such monomials $u \in \Mc$ is independent of $\ab$ in this case.

\smallskip

\noindent
{\bf Case 3}: $|\{ \ell_1, \ell_2, m_1, m_2\}|=4$. 

We show that, given positive integers $1 \le \mu_1 <  \mu_2 <  \mu_3 <  \mu_4 \le n$, 
the number of monomials $u \in \Mc$ of the form \eqref{std} such that 
$\{ \ell_1, \ell_2, m_1, m_2\} = \{\mu_1, \mu_2,  \mu_3,  \mu_4\}$ is independent of $\ab$. 
Then the possible patterns are listed in Table~\ref{Case 3}, 
where $\mu_2 < \ell_3$, $\mu_4 < m_3$, $\ell_3 \le m_3$, $\mu_3 < \ell_3'$, $\mu_4 < m_3'$ and $\ell_3' \le m_3'$. 
\begin{table}[h]
  \begin{tabular}{|c|c|c|}
\hline
   $w_{\mu_1} > w_{\mu_2} > w_{\mu_3}>w_{\mu_4}$ 
& $w_{\mu_1} > w_{\mu_2} > w_{\mu_3} <w_{\mu_4}$ 
& $w_{\mu_1} < w_{\mu_2} > w_{\mu_3}>w_{\mu_4}$ \\
\hline
$\begin{bmatrix}
\mu_1\\
\mu_2\\
\ell_3\\
\vdots
\end{bmatrix}
\begin{bmatrix}
\mu_3\\
\mu_4\\
m_3\\
\vdots
\end{bmatrix}$,
$\begin{bmatrix}
\mu_1\\
\mu_3\\
\ell_3'\\
\vdots
\end{bmatrix}
\begin{bmatrix}
\mu_2\\
\mu_4\\
m_3'\\
\vdots
\end{bmatrix}$
&
$\begin{bmatrix}
\mu_1\\
\mu_2\\
\ell_3\\
\vdots
\end{bmatrix}
\begin{bmatrix}
\mu_4\\
\mu_3\\
m_3\\
\vdots
\end{bmatrix}$,
$\begin{bmatrix}
\mu_2\\
\mu_3\\
\ell_3'\\
\vdots
\end{bmatrix}
\begin{bmatrix}
\mu_4\\
\mu_1\\
m_3'\\
\vdots
\end{bmatrix}$
&
$\begin{bmatrix}
\mu_2\\
\mu_1\\
\ell_3\\
\vdots
\end{bmatrix}
\begin{bmatrix}
\mu_3\\
\mu_4\\
m_3\\
\vdots
\end{bmatrix}$,
$\begin{bmatrix}
\mu_2\\
\mu_3\\
\ell_3'\\
\vdots
\end{bmatrix}
\begin{bmatrix}
\mu_4\\
\mu_1\\
m_3'\\
\vdots
\end{bmatrix}$\\
\hline
  \end{tabular}

\smallskip
  
\begin{tabular}{|c|c|c|}
\hline
   $w_{\mu_1} > w_{\mu_2} < w_{\mu_3}>w_{\mu_4}$ 
& $w_{\mu_1} > w_{\mu_2} < w_{\mu_3}<w_{\mu_4}$ 
& $w_{\mu_1} < w_{\mu_2} > w_{\mu_3}<w_{\mu_4}$ \\
\hline
$\begin{bmatrix}
\mu_1\\
\mu_2\\
\ell_3\\
\vdots
\end{bmatrix}
\begin{bmatrix}
\mu_3\\
\mu_4\\
m_3\\
\vdots
\end{bmatrix}$,
$\begin{bmatrix}
\mu_3\\
\mu_1\\
\ell_3'\\
\vdots
\end{bmatrix}
\begin{bmatrix}
\mu_4\\
\mu_2\\
m_3'\\
\vdots
\end{bmatrix}$
&
$\begin{bmatrix}
\mu_1\\
\mu_2\\
\ell_3\\
\vdots
\end{bmatrix}
\begin{bmatrix}
\mu_4\\
\mu_3\\
m_3\\
\vdots
\end{bmatrix}$,
$\begin{bmatrix}
\mu_3\\
\mu_1\\
\ell_3'\\
\vdots
\end{bmatrix}
\begin{bmatrix}
\mu_4\\
\mu_2\\
m_3'\\
\vdots
\end{bmatrix}$
&
$\begin{bmatrix}
\mu_2\\
\mu_1\\
\ell_3\\
\vdots
\end{bmatrix}
\begin{bmatrix}
\mu_4\\
\mu_3\\
m_3\\
\vdots
\end{bmatrix}$,
$\begin{bmatrix}
\mu_3\\
\mu_1\\
\ell_3'\\
\vdots
\end{bmatrix}
\begin{bmatrix}
\mu_4\\
\mu_2\\
m_3'\\
\vdots
\end{bmatrix}$\\
\hline
  \end{tabular}

\smallskip
  
  \begin{tabular}{|c|c|}
\hline
   $w_{\mu_1} < w_{\mu_2} < w_{\mu_3}>w_{\mu_4}$ 
& $w_{\mu_1} < w_{\mu_2} < w_{\mu_3}<w_{\mu_4}$ \\
\hline
$\begin{bmatrix}
\mu_2\\
\mu_1\\
\ell_3\\
\vdots
\end{bmatrix}
\begin{bmatrix}
\mu_3\\
\mu_4\\
m_3\\
\vdots
\end{bmatrix}$,
$\begin{bmatrix}
\mu_3\\
\mu_1\\
\ell_3'\\
\vdots
\end{bmatrix}
\begin{bmatrix}
\mu_4\\
\mu_2\\
m_3'\\
\vdots
\end{bmatrix}$
&
$\begin{bmatrix}
\mu_2\\
\mu_1\\
\ell_3\\
\vdots
\end{bmatrix}
\begin{bmatrix}
\mu_4\\
\mu_3\\
m_3\\
\vdots
\end{bmatrix}$,
$\begin{bmatrix}
\mu_3\\
\mu_1\\
\ell_3'\\
\vdots
\end{bmatrix}
\begin{bmatrix}
\mu_4\\
\mu_2\\
m_3'\\
\vdots
\end{bmatrix}$\\
\hline
  \end{tabular}
\caption{Case 3}
\label{Case 3}
\end{table}
Thus the number of such monomials $u \in \Mc$ is independent of $\ab$ in this case.

\smallskip

Next, we prove (ii). On the contrary, suppose that $\Mc$ is linearly dependent over $K$. 
Since the set of all quadratic monomials in $S \setminus \ini_<( J_{\Lambda_\ab} )$ is a basis of $ [A_\ab]_2$, 
there exists $u \in \Mc$ such that $u$ belongs to $\ini_<( J_{\Lambda_\ab})$. 
Then there exists a binomial $u -v \in {\mathcal G}_\ab$ such that $\ini_<(u-v)=u$ and $v \notin \ini_<( J_{\Lambda_\ab})$. 
Since $v \in (S \setminus \ini_<( J_{\Lambda_\ab}))  \subset (S \setminus \langle \ini_< (\Gc_\ab) \rangle)$, we have $v \in \Mc$. Let
\[u=
\begin{bmatrix}
\ell_1\\
\ell_2\\
\ell_3\\
\vdots\\
\ell_r
\end{bmatrix}
\begin{bmatrix}
m_1\\
m_2\\
m_3\\
\vdots\\
m_r
\end{bmatrix} \ \text{ and }\ 
v=
\begin{bmatrix}
\ell_1'\\
\ell_2'\\
\ell_3'\\
\vdots\\
\ell_r'
\end{bmatrix}
\begin{bmatrix}
m_1'\\
m_2'\\
m_3'\\
\vdots\\
m_r'
\end{bmatrix}.\]
Since $u$ and $v$ belong to $\Mc$, we may assume that $\ell_k \le m_k$ and $\ell_k' \le m_k'$ for each $k= 1,3,\ldots, r$.
Moreover since $u -v $ belongs to $J_{\Lambda_\ab}$, we have $\{\ell_k, m_k\} = \{\ell_k', m_k'\}$ for each $k= 1,2,\ldots, r$.
Thus $(\ell_k, m_k) = (\ell_k', m_k')$ for each $k= 1,3,\ldots, r$, and $(\ell_2, m_2) = (m_2', \ell_2')$ with $\ell_2 \neq m_2$. 
Since $\{\ell_1, \ell_2 , m_1, m_2\} = \{\ell_1', \ell_2' , m_1', m_2'\}$ holds, $u$ and $v$ must appear in the same case above.
However there is no such a pair of monomials in any case above, a contradiction. 
Thus $\Mc$ is linearly independent, and hence a basis of $[A_\ab]_2 $.
\end{proof}

\begin{proof}[Proof of Theorem~\ref{thm:toric_deg}]
Our goal is to show that $J_{\Lambda_\ab}=\ini_{w_{M_\ab}}(I_{r,n})$ by Theorem~\ref{thm:enough}. 

The proof is based on the same idea as that of \cite[Theorem 4.3]{CM}. 
By the way of the proof employed there, we see that we may prove the following: 
\begin{itemize}
\item $J_{\Lambda_\ab}$ is quadratically generated; 
\item $\dim([A_\ab]_2)=\dim([A_0]_2)$. 
\end{itemize}
Both assertions are the consequences of Theorem~\ref{thm:quad} and Lemma~\ref{standard_monomials}, respectively. 
\end{proof}
\begin{Remark}
Let us explain how our proof of $\dim([A_\ab]_2)=\dim([A_0]_2)$ becomes simpler than that of \cite[Theorem 4.3]{CM}. 
The proof of $\dim([A_\ab]_2)=\dim([A_0]_2)$ in \cite{CM} consists of three lemmas, \cite[Lemma 3.10--3.12]{CM}. 
Roughly speaking, the idea of the proof is 
to find a basis for $[A_\ab]_2$ and a bijection between a basis for $[A_\ab]_2$ and the well-known basis for $[A_0]_2$. 
For this, we need much more detailed case-by-case discussions than ours. 
On the other hand, we use the special property of initial ideals (\cite[Proposition 1.1]{Stu}). 

On the other hand, \cite[Theorem 4.1]{CM} only claims the quadratic generation of $J_{\Lambda_\ab}$, 
while Theorem~\ref{thm:quad} claims the stronger property. 
\end{Remark}

\begin{Corollary}\label{cor:toric_deg}
Each $s$-block diagonal matching field associated with $\ab=(a_1,\ldots,a_s) \in \ZZ_{>0}^s$ such that 
$\sum_{i=1}^s a_i=n$ and $a_i \in \{1,2\}$ for $i \not\in \{1,s\}$ gives rise to a toric degeneration of $\Gr(r,n)$. 
\end{Corollary}

We conclude the present paper with the following:
\begin{Question}[{cf. \cite[Remark 3.14]{MS}}]
Does every $s$-block diagonal matching field give rise to a toric degeneration of $\Gr(r,n)$? 
\end{Question}


\end{document}